\pdfoutput=1
\documentclass[11pt]{article}
\usepackage[T1]{fontenc}
\usepackage[sc]{mathpazo}
\usepackage{amsmath}
\usepackage{amsmath}		
\usepackage[includefoot]{geometry}
\usepackage{amsmath}
\usepackage{amssymb}
\usepackage{amsthm}
\usepackage{array}

\usepackage{tikz}
\usepackage{tikz-3dplot}
\usepackage{pgfplots}

\usepackage{tkz-graph}
\usetikzlibrary{arrows}
\usetikzlibrary{arrows.meta}
\usetikzlibrary{automata}
\usetikzlibrary{decorations.markings}
\usetikzlibrary{patterns}
\usetikzlibrary{positioning}
\usetikzlibrary{shapes.geometric}

\usepackage{pdfpages}

\usepackage{pifont}
\usepackage{forest}

\usepackage{multirow,multicol,enumitem}

\usepackage{ytableau}
\usepackage{bm}
\usepackage{bbm}

\usepackage{caption}
\usepackage{mathrsfs}

\theoremstyle{plain}
\newtheorem{theorem}{Theorem}
\newtheorem{corollary}[theorem]{Corollary}
\newtheorem{lemma}[theorem]{Lemma}

\newtheorem{conjecture}[theorem]{Conjecture}

\theoremstyle{definition}
\newtheorem{definition}[theorem]{Definition}
\newtheorem{example}[theorem]{Example}
\theoremstyle{remark}

\usepackage{titlesec}

\newcommand{\name}[1]{\noindent {\it #1}\medskip}

\begin{document}

\section*{Type B Set partitions, an analogue of restricted growth functions}

  \name{Amrita Acharyya}\\
  Department of Mathematics and Statistics, University of Toledo, Ohio 43606, USA\\
  *Correspondence: amrita.acharyya@utoledo.edu
    
\begin{abstract}
	In this work, we study type B set partitions for a given specific positive integer $k$ defined over $\langle n\rangle=\{-n, -(n-1),\cdots -1,0,1,\cdots n-1,n\}$. We found a few generating functions of type B analogue for some of the set partition statistics defined by Wachs, White and Steingrimsson for partitions over positive integers $[n] =\{1,2,\cdots n\}$, both for standard and ordered set partitions respectively. We extended the idea of restricted growth functions utilized by Wachs and White for set partitions over $[n]$, in the scenario of $\langle n\rangle$ and called the analogue as Signed Restricted Growth Function (SRGF). 
   We discussed analogues of major index for type B partitions in terms of SRGF. We found an analogue of Foata bijection and reduced matrix for type B set partitions as done by Sagan for set partitions of $[n]$ with sepcific number of blocks $k$.  We conclude with some open questions regarding the type B analogue of some well known results already done in case of set partitions of $[n]$.\\
   
   Key words: q-analogue, signed set partitions, Stirling number, generating functions, restricted growth functions\\
   2023 Mathematics Subject Classification: 05A05, 05A15, 05A18, 05A19
   
\end{abstract}
\section{Introduction}
In the preliminary section, we initially describe four fundamental statistics introduced by Wachs and White \cite{MR2445243} using the technique of restricted growth functions as in \cite{MR1544457}, where the number of blocks $k$ of the set partition of $[n]$ turns out to be the maximal letter in the restricted growth function as observed by Sagan in \cite{MR1644459}.
In the preliminary section, we initially describe the type-B analogue of ten set partition statistics over $\langle n\rangle$, which were defined by Steingrimsson over $[n]$ in case of standard and ordered both type of set partitions. We found the generating functions of the type B analogue of some of the set partition statistics defined over $\langle n\rangle$ which were defined by Steingrimsson over $[n]$, for any specified number of blocks $k$ in terms of $q$-Stirling numbers of the second kind. Stirling numbers of both kinds have been extensively studied in combinatorics and have interesting applications in algebra and geometry. However,
$q$-Stirling numbers in type B have appeared sporadically in the literature over the last several decades.
In \cite{MR1644458} Sagan and Swanson worked on various statistics over signed or type B partition using type B $q$-Stirling number of second kind.  Haglund, Rhoades, and Shimozono \cite {MR1644460} showed that there is a connection between ordered set partitions, generalized coinvariant algebras, and
the Delta Conjecture. In
related work, Zabrocki \cite {MR1644461} conjectured that the tri-graded Hilbert series
4
of the type A superdiagonal coinvariant algebra has coefficients which are
the ordered $q$-Stirling number of the second kind. Swanson and Wallach \cite {MR1644469} made a corresponding conjecture
in type B. This led them to conjecture that an alternating sum involving
these ordered Stirling numbers equals one. Sagan and Swanson proved that conjecture in \cite{MR1644458}.

There is a bijection between the set partitions of $[n]$ in standard form and restricted growth functions (RGF). Wachs and White defined four fundamental statistics on those RGFs. In the final section of this paper, we defined an analogue of restricted growth function in case of type B set partitions of $\langle n\rangle$  and called it as Signed Restricted Growth Function. We found a similar kind of bijection between set partitions of $\langle n\rangle$ and SRGF.
In \cite{MR1644457} Steingrimsson defined ten set partition statistics over the set partitions of $[n]$. Four of these were already defined by Wachs and White as above in case of standard set partitions, and their treatment was in terms of restricted growth functions, a different way of representing partitions in standard form only. Another four statistics
are mirror images of the aforementioned ones. The last two statistics, essentially defined by Foata
and Zeilberger for permutations, are in fact each equal to the difference of two
of the first eight statistics. 
In \cite{MR1644459} Sagan has shown that the Foata bijection 
interchanging inversion and major index for permutations also has a version for partitions of $[n]$. 
In section 3, we discussed an analogue of Foata bijection and using SRGF we have shown that is interchanging the inversion and major index for type B partitions over $\langle n\rangle$. 
In \cite{MR1644459} Sagan has given an interpretation of major index for set partitions of $[n]$ using reduced matrices.
 Here we discussed an analogue of such matrices for type B set partitions and have shown that the analogue of reduced matrix is preserving the major index as done by Sagan for set partitions of $[n]$. Utilizing the idea of two inversion vectors for RGF as done in \cite{MR1644459}, we discussed  eight vectors for SRGF corresponding to the type B analogue of Steingrimsson's statistics. 
\section*{Preliminary[1]}
\begin{definition}\cite {MR1544457} A restricted growth function (RGF) is a sequence $w = a_{1} ...a_{n}$ of positive integers subject to the
	restrictions 
	\begin{enumerate}
		\item $a_{1} =1$. 
		\item For $i\geq 2$,  $a_{i} \leq 1+max\{a_{1},...,a_{i-1}\}$
	\end{enumerate}

 \end{definition}
In \cite {MR1544457} a partition of $[n]$ is written as $\sigma = B_{1}/\cdots /B_{k}\vdash S$ where the subsets $B_{i}$ are called blocks. We use the notation $\Pi_{n} =\{\sigma : \sigma \vdash [n]\}$.
In order to connect set partitions with the statistics of Wachs and White, they are converted into restricted growth functions as in \cite {MR1544457} and . That requires the elements of $\Pi_{n}$ in standard form.
\begin{definition}
	We say $\sigma = B_{1}/\cdots /B_{k} \in \Pi_{n}$ is in standard form if min $B_{1} < \cdots <$ min $B_{k}$. Thus it follows that min $B_{1} =1$.
\end{definition}

We assume all partitions in $\Pi_{n}$ are written in the standard form. Associate with $\sigma\in\Pi_{n}$ the word $w(\sigma) = a_{1}\cdots a_{n}$ where $a_{i} =j$ if and only if $i\in B_{j}$. For example  $w(16/23478/5) =12223122.$
Let, $\Pi_{n,k}$ be the set of all words in $\Pi_{n}$ with exactly $k$ many blocks. $R_{n} =\{w : w$ is an RGF of length $n\}$. Let, $R_{n,k} =\{w : w$ is an RGF of length $n$ with maximal letter $k\}$. 

The four statistics of Wachs and White are denoted as $lb, ls, rb$ and $rs$ where \lq\lq l\rq\rq stands for \lq\lq left\rq\rq, \lq\lq r\rq\rq stands for \lq\lq right\rq\rq, \lq\lq b\rq\rq stands for \lq\lq bigger\rq\rq, and \lq\lq s\rq\rq stands for \lq\lq smaller\rq\rq. The left-bigger statistic is described and the other three should become clear by analogy. Given a word $w =a_{1}\cdots  a_{n}$  define $lb(a_{j}) =\#\{a_{i} : i < j$ and $a_{i} > a_{j}\}$.
 It is important to note that, the cardinality of a set is taken, so if there are multiple copies of such an integer then it is only counted once. Also, clearly $lb(a_j)$ depends on the word containing $a_j$, not just $a_j$ itself. By way of example, if $w =12332412$, then $lb(a_7) =3$. Define $lb(w) =lb(a_1)+\cdots +lb(a_n)$. Continuing the above example, $lb(12332412) =0+0+0+0+1+0+3+2=6$. To simplify notation, $lb(\sigma)$ is taken instead of more cumbersome $lb(w(\sigma))$. Accordingly, $ls(\sigma), rb(\sigma), rs(\sigma)$ are defined.

 Now let, $O\Pi_{n}$ be the set of all  ordered partitions of $[n]$ (that means the set partitions are not necessarily in standard form).
 Let, $O\Pi_{n,k}$ be the set of all words in $O\Pi_{n}$ with exactly $k$ blocks.

In order to define the ten statistics, Steingrimsson first defined the openers and closers of the blocks for any $\pi \in  O\Pi_{n,k}$. The opener of a block is its least element and the closer is its greatest element.

\begin{definition}[Steingrimsson]
	Given a partition $\pi \in O\Pi_{n,k}
	$ let open $\pi$ and clos $\pi$ be the set of
	openers and closers, respectively, of $\pi$. Let, block(i) be the number (counting from
	the left) of the block containing the letter i. Eight coordinate statistics
	are defined as follows:
	\begin{enumerate}
	\item $ros_{i}\pi$ =  $\# \{j | i > j, j \in open \pi, block(j) > block(i)\},$
	\item $rob_{i}\pi$ =  $\#\{j | i < j, j \in open \pi, block(j) > block(i)\},$
	\item $rcs_{i}\pi$ =  $\#\{j | i > j, j \in clos \pi, block(j) > block(i)\},$
	\item $rcb_{i}\pi$ =  $\#\{j | i < j, j \in clos \pi, block(j) > block(i)\},$
	\item $los_{i}\pi$ =  $\#\{j | i > j, j \in open \pi, block(j) < block(i)\},$
	\item $lob_{i}\pi$ = $\#\{j | i < j, j \in open \pi, block(j) < block(i)\},$
	\item $lcs_{i}\pi$ =  $\#\{j | i > j, j \in clos \pi, block(j) < block(i)\},$
	\item $lcb_{i}\pi$ =  $\#\{j | i < j, j \in clos \pi, block(j) < block(i)\}$.
\end{enumerate}
	The hash tags denote the cardinalities of the corresponding sets. Moreover, let $rsb_i$ be the number of blocks $B$, to the right of the block containing
	$i$ such that the opener of $B$ is smaller than $i$ and the closer of $B$ is greater than
	$i$ (rsb is an abbreviation for right, smaller, bigger). Also $lsb_i$ is defined in an
	analogous way, with right replaced by left.
	Set 	\textbf{$ros \pi = 
	\sum_{i}ros_{i} \pi$}
	and likewise for the remaining nine statistics, i.e. each of $rob, rcs, rcb, los, lob, lcs,
	lcb, rsb, lsb$ is defined to be the sum over all $i$ of the respective coordinate statistics.
\end{definition}
		 Defining $ROS\pi = ros\pi + \binom{k}{2}$ (and $RCB, LOS, LCB$ similarly), we have  \begin{theorem}(Steingrimsson)\cite{MR1644457} 
	 	$ROS, RCB, LOS, LCB$ are Euler-Mahonian on ordered partitions, that is, 
	 	$\sum_{\pi\in O\Pi_{n,k}}q^{ROS\pi}
	 	= [k]! S_{q}(n, k)$ and the same for the other three Statistics.
	 \end{theorem}
	 where $[k]!= [k][k-1][k-2]\cdots[1]$ with $[k] = [k]_{q}$ (we drop the $q$ from the suffix to make the notation simpler)$= 1+q+q^{2}+q^{3}+\cdots q^{k-1}$ and $S_{q}(n, k)$ is the $q$-Stirling numbers of second kind which can be described as in Lemma 1 in \cite{MR1644457}, $S_{q}(n, k) = q^
	 {k-1}
	 · S_{q}(n-1, k-1) + [k] · S_{q}(n-1, k)$.
	 
	 \section{Steingrimsson's Statistics for type B partitions and some generating functions in terms of $q$-Stirling numbers[2]}

\begin{definition}\cite{MR1644458}
The type B Stirling numbers of the second kind are defined	by the following recurrence relation:\\
$S_{B}(n, k) = S_{B}(n-1, k-1) + (2k + 1)S_{B}(n-1, k)$ and $S_{B}(0, k) = \delta_{0,k}$ (Kronecker
delta) The ordered version of $S_{B}(n, k)$ is $S^{o}_{B}(n,k)
 = (2k)!!S_{B}(n, k)$
\end{definition}

\begin{definition} \cite{MR1644458}
	The type B $q$-Stirling numbers of the second kind are defined by replacing the above recurrence relation with\\
	$S_{B}[n, k] = S_{B}[n-1, k-1] + [2k + 1]S_{B}[n-1, k]$.
	The ordered version of $S_B[n, k]$ is
	$S^{o}_{B}[n,k]
	= [2k]!!S_{B}[n, k]$
\end{definition}
	  where $[k]!!=[k][k-2][k-4]\cdots$ ending at $[2]$ or $[1]$ depending on $k$ is even or odd respectively.
	
\begin{definition} (Sagan and Swanson)\cite{MR1644458}
	A signed or type B partition is a partition of the set $\langle n \rangle = \{-n,\cdots -1,0,1,\cdots n\}$ of
	the form $\rho = S_{0}/S_{1}\cdots /S_{k}$, (We write
	$\rho \vdash_{B} \langle n \rangle $)
	satisfying
	\begin{enumerate}
	\item [1] $0\in S_{0}$ and if $i \in S_{0}$,  then $\bar{i} \in S_{0}$, and
	\item [2] for $i \geq 1$ we have $S_{2i} = -S_{2i-1}$,
	where $-S = \{-s|s\in S\}$. \\ 
	\end{enumerate}
\end{definition}	
Let $|S| = \{|s| : s \in S\}$, so that $|S_{2i}|
 = |S_{2i-1}|$ for $i \geq 1$. For all $i$ we let
$m_{i} = min |S_{i}|$. Let $S_B(\langle n \rangle, k)$ denote the set of all type B partitions of $\langle n \rangle$
with $2k + 1$ blocks in standard form.
We will always write signed partitions in standard form which means that
\begin{enumerate}
	\item [3] $m_{2i} \in S_{2i}$ for all $i$, and
	\item [4] $0 = m_{0} < m_{2} < m_{4} <\cdots < m_{2k}$
	
\end{enumerate}
\begin{definition}(Sagan and Swanson) \cite{MR1644458} An inversion of $\rho \vdash_{B} \langle n \rangle $ written in Standard form is a pair $(s, S_{j})$ satisfying
\begin{enumerate}
	\item [1] $s\in S_{i}$ for some $i<j$ and 
	\item[2] $s\geq m_{j}$
\end{enumerate}
Let Inv $\rho$ be set of inversions of $\rho$ and inv $\rho = \# Inv {\rho}$
\end{definition}
\begin{theorem}(Sagan and Swanson)\cite{MR1644458}
$S_{B}[n,k]= \sum_{\rho \in S_{B}(\langle n \rangle , k) } q^{inv{\rho}}$
\end{theorem}
\begin{definition}
	An ordered signed partition of $\langle n \rangle$ is a sequence $\omega =\\
	(S_0/S_1/S_2/ . . . /S_{2k})$ satisfying the first two conditions in the definition of signed or type B partition. Note
	that no assumption is made about standard form. The set of such partitions
	with $2k + 1$ blocks is denoted as $S^{o}_{B}
	(\langle n \rangle, k)$.
	The definition of inversion remains unchanged.
\end{definition}  
\begin{theorem}(Sagan and Swanson)\cite{MR1644458}
	For $n,k \geq 0, S^{o}_{B}[n,k]= \sum_{\omega \in S^{o}_{B}(\langle n \rangle, k)}q^{inv {\omega}}$
\end{theorem}

Note that defining inversion over an ordered signed partition of $[n]$ in the above way, matches with Steingrimsson's ros, while the same is applied over any $\pi \in O\Pi_{n,k}$.
For example consider $\pi = 47/3/159/68/2 \in O\Pi_{9,5}$. 
inv $\pi =\\ \#\{(4,B_{2}), (4,B_{3}), (4,B_{5}), (7,B_{2}), (7,B_{3}), (7,B_{4}), (7, B_{5}), (3, B_{3}), (3, B_{5}), (5, B_{5}), (9, B_{4})\\, (9, B_{5}), (6, B_{5}), (8, B_{5})\} =14 = ros \pi$.\\
This is the motivation to define in this work nine more statistics over any ordered signed partitions of $\langle n \rangle$, so that they matches with Steingrimsson's nine other statistics accordingly, while the same is applied over any $\pi \in O\Pi_{n,k}$.\\ To do this, we can further define $M_{i}= max|S_{i}|$, like $m_{i}$. For any $\rho \in S^{o}_{B}
(\langle n \rangle, k)$, noting that the negetive elements and $0$ in the type B partition of $\langle n\rangle$ does not contribute in $inv\rho$ we define the following. 

\begin{definition}

\begin{enumerate}
	\item $ros_{B}\pi =\#\{(s,S_{j})| s \in S_{i}$ for some $i<j$ and $s \geq m_{j}\},$
	\item $rob_{B}\pi =\#\{(s,S_{j})| s \in S_{i}$ for some $i<j$ and $s \leq m_{j}$ and $s>0\},$
	\item $rcs_{B}\pi =\#\{(s,S_{j})| s \in S_{i}$ for some $i<j$ and $s \geq M_{j}\},$
	\item $rcb_{B}\pi =\#\{(s,S_{j})| s \in S_{i}$ for some $i<j$ and $s \leq M_{j}$ and $s>0\},$
	\item $los_{B}\pi =\#\{(s,S_{j})| s \in S_{i}$ for some $i>j>0$ and $s \geq m_{j}\},$
	\item $lob_{B}\pi =\#\{(s,S_{j})| s \in S_{i}$ for some $i>j>0$ and $s \leq m_{j}$ and $s>0\},$
	\item $lcs_{B}\pi =\#\{(s,S_{j})| s \in S_{i}$ for some $i>j>0$ and $s \geq M_{j}\},$
	\item $lcb_{B}\pi =\#\{(s,S_{j})| s \in S_{i}$ for some $i>j>0$ and $s \leq M_{j}$ and $s>0\},$
	\item $rsb_{B}\pi =\#\{(s,S_{j})| s \in S_{i}$ for some $i<j$ and $m_{j}\leq s \leq M_{j}\},$
	\item $lsb_{B}\pi =\#\{(s,S_{j})| s \in S_{i}$ for some $i>j>0$ and $m_{j}\leq s \leq M_{j}\}$
\end{enumerate}

\end{definition}
As an example consider the same $\pi$ above. Note that by the above definition $los_{B}\pi = \{(5,B_{1}), (5,B_2), (9,B_{1}), (9, B_{2}), (6, B_{1}), (6, B_{2}), (6, B_{3}), (8, B_{1}), (8,B_{2}), (8, B_{3}), (2, B_{3})\}= 11$, where as by Steingrimsson $los{\pi} = 0+0+0+0+2($corresponding to $(5,B_{1}), (5,B_2))+2($corresponding to $(9,B_{1}), (9,B_2))+3($corresponding to$ (6,B_{1}), (6,B_2), (6, B_{3}))+\\3($
corresponding to$ (8,B_{1}), (8,B_2), (8, B_{3}))+1($corresponding to $(2,B_{3}),)=11$.\\
Note in \cite{MR1644457}, given a partition $\pi$ of $[n]$, let $\pi^{c}$
be the partition obtained by complementing
each of the letters in $\pi$, that is, by replacing the letter $i$ by $n+1-i$. Then it follows that $rcb \pi^{c}
= ros \pi$ and that $rcs \pi^{c}
 = rob \pi$. In order to have similar result for type B partitions in $S^{0}_{B}(\langle n \rangle, k)$, we define the complement of any $\pi \in S^{0}_{B}(\langle n \rangle, k)$ in the following way:
 \begin{definition}
 	For any  $\pi \in S^{0}_{B}(\langle n \rangle, k), \pi^{c}$ is obtained by replacing any positive $i$ by $n+1-i$, and $\bar{i}$ by $\overline{n+1-i}$ and keeping $0$ the same.	
 \end{definition}
 
 Then it follows that $rcb \pi^{c}
 = ros \pi$ and that $rcs \pi^{c}
 = rob \pi$, as because each $(s, S_{j})$ contributing in $ros \pi$ gives $(n+1-s, S_{j})$ contributing in $rcb \pi^{c}$ and conversely.

 Although, as in \cite{MR1644457}  where every right statistic
 is equidistributed with its corresponding left statistic (since reversing the order of the blocks in an ordered partition turns
 a left opener into a right opener and likewise for closers), the exact same is not the case for type B partitions. We have 
 \begin{theorem}
 	 $q^{k(k+1)}S_{B}[n,k]=\sum_{\rho\in S_{B}(\langle n\rangle,k)} q^{los_{B'}\rho}$ where $los_{B'}\rho= los_{B}\rho$ dropping the condition $j>0$ in the definition of $los_{B}\rho$. 	
 \end{theorem}
 \begin{proof} We follow the idea of the proof of theorem 4 in \cite {MR1644457} and theorem 3.7 in \cite {MR1644458}.
 	The proof follows by induction on $n$.\\
 	Base case: If $n=1$, then there are two possibilities about $k$. $k=0,k=1$. If $k$=0, then $2k+1= 1$ and the only element of $S_{B}(\langle 1\rangle, 0)$ to consider is $01\bar{1}$ which gives the result. If $k=1$, then $2k+1=3$. The only set partition to consider is $0/\bar{1}/1$ giving $los_{B'}$ as 2 and hence giving the result.\\
 	Now suppose the result be true for some $n-1$. 
 	 Given
 	 $\rho \in S_{B}(\langle n \rangle, k)$ we can remove $n$ and $-n$ to obtain a new partition $\rho^{\prime}$.\\
 	  	  	 If $n$ (and thus $-n$) is in a singleton block then $\rho^{\prime}\in
 	  S_{B}(\langle {n-1}\rangle, k-1)$
 	 and there is only one way to construct $\rho$ from $\rho^{\prime}$.
 	 	  Furthermore, in this case
 	 the standardization condition forces $S_{2k-1} = \{-n\}$ and $S_{2k} = \{n\}$ in $\rho$. It follows the $los_{B'}\rho=los_{B'}\rho^{\prime}+2k$. So, by induction such $\rho$ contributes $q^{(k-1)k}.q^{2k}S_{B}[n-1,k-1]=q^{k(k+1)}S_{B}[n-1,k-1]$. If $n$ and $-n$ are in a block with other elements, then $\rho^{\prime}\in S_{B}(\langle n-1 \rangle, k)$ which induces $i$ many $(n,S_{j})$, namely $(n, S_{0}), (n, S_{1}),\cdots (n,S_{i-1})$ elements adding to previous $los_{B'}$ and thus for any such $\rho, los_{B'}(\rho)=los_{B'}(\rho^{\prime})+i$ $\forall i$ with $0\leq i\leq 2k$. Thus the contribution of these $\rho$ are $[2k+1]q^{k(k+1)}S_{B}[n-1,k]$. Hence, we are done.
 	\end{proof} 
 
 \begin{theorem}
 We have $q^{k}S^{0}_{B}[n,k]= \sum_{\rho\in S^{0}_{B}(\langle n\rangle,k)} q^{los_{B'}\rho}$	
 \end{theorem}
 \begin{proof}
 	We follow the idea of the proof of theorem 4 in \cite {MR1644457} and theorem 3.7 in \cite {MR1644458} along with the above theorem.
 	The proof follows by induction on $n$\\
 	Base case: For $n=1,k=1$ the result holds.
 	For the rest we take the same approach as in the proof of the last theorem. Given
 	$\rho \in S_{B}(\langle n \rangle, k)$ we can remove $n$ and $-n$ to obtain a new partition $\rho^{\prime}$.\\
 	If $n$ (and thus $-n$) is in a singleton block then $\rho^{\prime}\in
 	S_{B}(\langle {n-1}\rangle, k-1)$ and now $n$ can stay in any block except $S_{0}$ adding $1,2,3,\cdots 2k$ respectively. Thus these type of $\rho^{\prime}$ gives all together by induction hypothesis\\ $los_{B'}\rho^{\prime}= (q+q^{2}+q^{3}+\cdots q^{2k})[2(k-1)]!!q^{k-1}S_{B}[n-1,k-1] = q^{k}S^{0}_{B}[n-1,k-1]$.
 	Now if $n$ and $-n$ are in a block with other elements, then $\rho \in S_{B}(\langle n-1 \rangle, k)$ and as in the end of the proof of the last theorem using induction hypothesis these type of $\rho$ all together contributes $[2k]!!S_{B}[n-1,k]q^{k}[2k+1]$.
 	Hence the result follows, as $S^{0}_{B}[n,k]= [2k]!!S_{B}[n,k]$.\\
 \end{proof}
 
 \begin{lemma}
 	Let $A_k$ be the set of standard type B partitions on $2k+1$ blocks and $Z_k$ be the set of ordered type B partitions on $2k+1$ blocks with $lob_{B'}(\pi) = 0$ for all $\pi \in Z_k$.
 	Then there exists a bijection between $A_k$ and $Z_k$.
 	Additionally, we have that $lob_B(\pi) = k$ for all $\pi \in A_k$, where $lob_{B'}\rho= lob_{B}\rho$ dropping the condition $j>0$ in the definition of $loB_{B}\rho$.
 \end{lemma}
 
 \begin{proof}
 	Suppose that $\pi \in A_k$.
 	As $lob_{B'}(\pi)$ is the cardinality of the set $lob_{B'}(\pi) = \{(s,S_j) : s \in S_i \text{ for some } i > j \text{ and } 0 < s \leq m_j\}$ and $s \geq m_i$ for all $s \in S_i$, we know that $(s,S_j) \in Lob_{B'}(\pi)$ implies $m_i \geq m_j$ for some $i > j$.
 	Since $\pi$ is a standard type B partition, this occurs exactly for the case where $i = 2\ell$ and $j = 2\ell-1$ for $1 \leq \ell \leq k$.
 	The only entry $s$ of $S_{2\ell}$ that satisfies $s \leq m_{2\ell-1}$ is of course $s = m_{2\ell} = m_{2\ell-1}$.
 	Thus $(m_{2\ell}, S_{2\ell-1})$ is an element of $lob_{B'}(\pi)$ for $1 \leq \ell \leq k$.
 	Hence $lob_{B'}(\pi) = k$.
 	
 	Now, suppose $\pi \in Z_k$.
 	As $lob_{B'}(\pi) = 0$, there is no $(s,S_j)$ such that $m_i \geq m_j$ and $s > 0$ for some $i > j$.
 	This implies that $0 < m_1 = m_2 < m_3 = m_4 < \dots < m_{2k-1} = m_{2k}$.
 	Additionally, we know that $m_{2\ell} \in S_{2\ell-1}$ for $1 \leq \ell \leq k$.
 	Otherwise, we would have $(m_{2\ell}, S_{2\ell-1}) \in Lob_{B'}(\pi)$ as before.
 	Thus $\pi$ is simply a standard type B partition with every $S_{2\ell}$ swapped with $S_{2\ell-1}$.
 	Hence, there is a bijection between $A_k$ and $Z_k$.
 \end{proof}
 
 \begin{corollary}
 	The generating function of $lob_{B'}$ over the standard type B partitions is given by
 	$$\sum_{\pi \in S_B[\langle n \rangle ,k]} q^{lob_{B'}(\pi)} = S_B[n,k] q^k.$$
 	Additionally, the generating function of $lob_{B'}$ over the ordered type B partitions satisfies the following
 	$$\sum_{\pi \in S_B^0[\langle n \rangle ,k]} q^{lob_{B'}(\pi)} = S_B[n,k].$$
 	In particular, 
 	$$\sum_{\pi \in S_B^0[\langle n \rangle ,1]} q^{lob_{B'}(\pi)} = S_B[n,1]q + S_B[n,1].$$
 \end{corollary}
 
 \begin{lemma}
 	The statistics $ros_B$ and $rcb_B$ are equidistributed over the ordered type B set partitions.
 	The statistics $rob_B$ and $rcs_B$ are equidistributed over the ordered type B set partitions.
 \end{lemma}
 \begin{proof}
 	The proof follows as
 $rcb \pi^{c}
 = ros \pi$ and that $rcs \pi^{c}
 = rob \pi$, as because each $(s, S_{j})$ contributing in $ros \pi$ gives $(n+1-s, S_{j})$ contributing in $rcb \pi^{c}$ and conversely.
\end{proof}

\begin{lemma} \label{lem-standard-complement-bijection}
	Let $f(\pi)$ be the function given by taking the standardization of $\pi^c$ for any $\pi \in S_B(\langle n \rangle, k)$.
	Then $f$ is a bijection of $S_B(\langle n \rangle, k)$ with $f^2(\pi) = \pi$.
\end{lemma}

\begin{conjecture}
    We have that $rcb_B(f(\pi)) = ros_B(\pi) + k(k-1)$ for all $\pi \in S_B(\langle n \rangle, k)$.
    This gives us that the generating function of $rcb_B$ over the standard type B set partitions is given by
    $$\sum_{\pi \in S_B(\langle n \rangle ,k)} q^{rcb_B(\pi)} = q^{k(k-1)}S_B[n,k].$$
\end{conjecture}

\begin{proof}
	Let $T$ be a part of $f(\pi)$ for some $\pi \in S_B(\langle n \rangle, k)$.
	Then $T$ is the complement of $S_i$ for some $S_i$ that is a part of $\pi$.
	This forces the complement of $T$ to be $S_i$.
	Hence $f^2(\pi)$ is a standard partition with the same parts as $\pi$.
	As $\pi$ was already a standard partition, we have $f^2(\pi) = \pi$ and that $f$ is a bijection.
\end{proof}

 \section{Signed Restricted Growth Functions[3]}
 
 \begin{definition}(SRGF):	
 	
 	A Signed Restricted Growth Function is a sequence of the form\\
 	$w=a_{0}a_{1}a^{*}_{1}a_{2}a^{*}_{2}\cdots a_{n}a^{*}_{n}$ of length $2n+1$, where
 	\begin{enumerate}
 		\item $a_{0}=0$
 		\item If $a_{i}=j$, then $a^{*}_{i}=\bar{j}$ and conversely $\forall i, j\in \{1,2\cdots n\}$.
 		\item $\forall i\geq 1, |a_{i+1}|\leq 1+Max\{|a_{0}|,|a_{1}|, |a_{2}|,\cdots |a_{i}|\}$
 		\item The pair $\bar{j}j$ appears before $j\bar{j}$ (if there is a $j\bar{j}$ in the sequence) for any $j\in \{1,2\cdots n\}$
 	\end{enumerate}
 \end{definition}
 Calling the set of all such SRGF of length 2n+1 as $SR_{n}$  and $SR_{n,k}$ accordingly for SRGF of length 2n+1 with maximal letter $k$, we can show analogously as in RGF that there is a bijection between the set of all type B partitions of $\langle n\rangle$ and $SR_{n}$, which preserves the one to one correspondence between $SR_{n,k}$ and $S_{B}(\langle n\rangle,k)$. The bijection is the following: Consider any type B partition $\pi$ of $\langle n\rangle$. Associate with $\pi$, the word $w(\pi)= a_{0}a_{1}a^{*}_{1}a_{2}a^{*}_{2}\cdots a_{n}a^{*}_{n}$, where 
 \begin{enumerate}
 	\item $a_{0}=0, a_{i}=a^{*}_{i}=0$, iff $i,\bar{i}\in S_{0}$, for $i\in \{1,2\cdots n\}$
 	\item $a_{i} = \bar{j},a^{*}_{i}=j$ iff $i\in S_{2j} $ for $i,j\in \{1,2\cdots n\}$
 	\item $a_{i} = j,a^{*}_{i}=\bar{j}$ iff $i\in S_{2j-1} $ for $i,j\in \{1,2\cdots n\}$
 	
 \end{enumerate}
 Note that for any type B partition $\pi$ of $\langle n \rangle, w(\pi)$ satisfies the condition i. and ii. in the definition of SRGF. Due to standardization of $\pi, w(\pi)$, satisfies condition iii. and iv. 
 \begin{example}\item For $\pi = 02\bar{2}/\bar{1}7/1\bar{7}/\bar{3}\bar{6}/36/\bar{4}5/4\bar{5},w(\pi)=0\bar{1}100\bar{2}2\bar{3}33\bar{3}\bar{2}21\bar{1}$
 \item For $w=000\bar{1}100\bar{2}21\bar{1}2\bar{2}\bar{3}3\bar{4}44\bar{4}$, the corresponding $\pi$ is 
 $01\bar{1}3\bar{3}/\bar{2}5/2\bar{5}/\bar{4}6/4\bar{6}/\bar{7}/7/\bar{8}9/8\bar{9}$
 \end{example}
 
 \textbf{Questions}:
 \begin{enumerate}
 	\item Is there any way to find the generating functions of the type B-analogue of the ten statistics due to Steingrimsson for set partitions of $\langle n\rangle$ in standard form by using the correspondence with SRGF as in the approach in \cite{MR1544457} or for some other known statistics like $\widehat{maj}$ (dual major index) for set partitions of $\langle n\rangle$ in standard form as in standard set partitions of $[n]$ in \cite{MR1644459}?
 	\item  What are the generating functions of the rest of the B-analogues of Steingrimsson's statistics for ordered set partitions of $\langle n\rangle$? 
 \end{enumerate}
 
 In reference \cite{MR1644459} section 4, we observe that, if for a set partition $\pi$, a positive integer contributes in the descent set of $F(\pi)$, where $F(\pi)$ is the Foata bijection as defined in \cite{MR1644459}, then the number of it's contribution is the number of it's corresponding occurrence in the inversion set of $\pi$. For example, in \cite{MR1644459} section 4, for $\pi= 138/2/476/59$  the positive integer $7$ contributes $1$ in the descent of $F(
\pi)= 1367/2/48/59$, which is the number of it's occurrence in the inversion set of $\pi$, namely as $(7,B_{1})$. This observation motivates us to define an analogue of Foata bijection for type B set partitions in $S_{B}(\langle n \rangle, k)$ as follows:\\
 This bijection $F$ is analogously defined via induction on $n$, as $F$ is identity whenever $n=0$. If $\pi \in S_{B}(\langle n \rangle, k)$, for $n>1$, let $\pi^{\prime}=\pi$ with $n, \bar{n}$ deleted. We construct $\sigma = F(\pi)$ from $\sigma^{\prime}= F(\pi^{\prime})$ as follows. If \begin{enumerate}
 	\item $\pi = S_{0}/S_{1}/S_{2}/\cdots/S_{2k}$, with $S_{2k-1}=\{\bar{n}\}$ and $S_{2k}= \{n\}$, (due to standardization condition, the other way can't happen), then let 
  $\sigma = \sigma^{\prime}$ with $\bar{n}$ and $n$ added in $S_{2k-1}$ and in $S_{2k}$ as singleton blocks respectively. 
 \item If $n$ is strictly contained in the block $S_{2k}$, then let $\sigma =\sigma^{\prime}$ with $n$ added in the block $S_{2k}$ and $\bar{n}$ added in the block $S_{2k-1}$.
 \item  If $n$ is in $S_{2k-1}$ and if $\bar{n}$ is in $S_{2k}$, then let $\sigma =\sigma^{\prime}$, along with both $n, \bar{n}$ added in $S_{0}$. 
 \item If $n$ or $\bar{n}$ are contained in $S_{2i}$, where $0<i<k$, then $\sigma =\sigma^{\prime}$ with $n$ or $\bar{n}$ added in $S_{2(k-i)-1}$ and $S_{2(k-i)}$ respectively in a way so that the mutual ordering is flipped. 
  \item If $n, \bar{n}$ are in $S_{0}$, then $n$ is added in $S_{2k-1}$ and $\bar{n}$ is added in $S_{2k}$.
\end{enumerate}
 So, we consider the following example where $\pi = 02\bar{2}/\bar{1}7/1\bar{7}/\bar{3}\bar{6}/36/\bar{4}5/4\bar{5}$
 
 \textbf{Table for the bijection F}\\
 \begin{tabular}{|l|c|c|}
 	\hline
 	$n$ & $\pi$ & $\sigma=F(\pi)$ \\
 	\hline
 	$0$ & $0$ & $0$ \\ 	
 	\hline
 	$1$ & $0/\bar{1}/1$ & $0/\bar{1}/1$ \\
 	\hline
 	$2$ & $02\bar{2}/\bar{1}/1$ & $0/\bar{1}2/1\bar{2}$ \\
 	\hline
 	$3$ & $02\bar{2}/\bar{1}/1/\bar{3}/3$   &	$0/\bar{1}2/1\bar{2}/\bar{3}/3$ \\
 	\hline
 	$4$ & $ 02\bar{2}/\bar{1}/1/\bar{3}/3/\bar{4}/4$ &$0/\bar{1}2/1\bar{2}/\bar{3}/3/\bar{4}/4$	\\
 	\hline
 	$5$ & $ 02\bar{2}/\bar{1}/1/\bar{3}/3/\bar{4}5/4\bar{5}$ & $05\bar{5}/\bar{1}2/1\bar{2}/\bar{3}/3/\bar{4}/4$\\
 	\hline
 	$6$ & $ 02\bar{2}/\bar{1}/1/\bar{3}\bar{6}/36/\bar{4}5/4\bar{5}$ & $05\bar{5}/\bar{1}26/1\bar{2}\bar{6}/\bar{3}/3/\bar{4}/4$\\
 	\hline
 	$7$ & $ 02\bar{2}/\bar{1}7/1\bar{7}/\bar{3}\bar{6}/36/\bar{4}5/4\bar{5}$ &	$05\bar{5}/\bar{1}26/1\bar{2}\bar{6}/\bar{3}\bar{7}/37/\bar{4}/4$\\
 	\hline
 	
 \end{tabular}
 
 We note that inv $(\pi)$ = maj$(F(\pi))$ =$10$ and inv$(F(\pi))$ = maj$(\pi)$ =$14$. We have the following theorem as an analogue of theorem 4.1 in \cite{MR1644459}:
 \begin{theorem}
 	The map $F:S_{B}(\langle n\rangle, k)\mapsto S_{B}(\langle n\rangle, k)$ defined above is a bijection where $\forall \pi \in S_{B}(\langle n\rangle, k)$ \begin{enumerate}
 		\item inv$(\pi)$ = maj$(F(\pi))$
 		\item inv$(F(\pi))$ = maj$(\pi)$
 	\end{enumerate}
 \end{theorem}
 As an analogue of section 2 in\cite{MR1644459}, we define $lb$ vector and $ls$ vector for SRGF induced by type B set partitions and denote them as $lb_{B}, ls_{B}$ respectively. We further extend the idea of major index as in section 2 in\cite{MR1644459}, via $lb_{B}$ for any element in $SR_{n,k}$.
 
 We define an $lb_{B}$ vector for SRGF as follows:
 
 \begin{definition}Let $w= a_{0}a_{1}a^{*}_{1}a_{2}a^{*}_{2}\cdots a_{k}a^{*}_{k}$ be an SRGF. Then,
 \begin{enumerate}
 	 \item $lb_{B}(a_{l})=0,$ if $l=0$ or if $a_{l}=\bar{j}$ for some $j>0$. 
 	 \item If we have an occurrence of $00$ after nonzero digits, each $0$ contributes $m$, where $m$ is the number of  $j$ to it's left, so that $|j|>0$.
 	 \item If $j\bar{j}$ appears after $\bar{j}j$, then the later $j>0$ contributes $1+2m$ where $m$ is the number of distinct $i$ to the left of that $j$, so that $i>j$. 
 	 \item If $\bar{j}j$ has repeated occurrences, the later $j$ contributes $2m$, where $m$ is the number of distinct $i$ to the left of that $j$, so that $i>j$. 
 	 \item If we get $l\bar{l}$ to the right of $j\bar{j}$ or $\bar{j}j$, where $l<j$, then $l$ contributes $2m+1$ where $m$ is the number of $j$ to it's left such that $j>l$.
 	 \item If we get $\bar{l}l$ to the right of $j\bar{j}$ or $\bar{j}j$, where $l<j$, then $l$ contributes $2m$ where $m$ is the number of $j$ to it's left such that $j>l$.
 	  	\end{enumerate}
 	  	\end{definition}
 \begin{example}Let $\pi = 0\bar{1}13\bar{3}/\bar{2}4/2\bar{4}/\bar{5}/5/\bar{6}\bar{8}/68/\bar{7}/7$.
 The corresponding SRGF is\\ $w(\pi)=000\bar{1}1001\bar{1}\bar{2}2\bar{3}3\bar{4}4\bar{3}3$. The corresponding $lb_{B}$ vector is
 $00000111000000002$ giving us the $lb_{B}$ statistics (as one of the four fundamental statistics of Wachs and White) as the sum of the digits in the vector as $5$ which is the same as the inversion of $\pi$.
 \end{example}
 Observing this we define the major index of any such SRGF $w=w(\pi)=\\ a_{0}a_{1}a^{*}_{1}a_{2}a^{*}_{2}\cdots a_{k}a^{*}_{k}$ analogously as in \cite{MR1644459} as follows:
 \begin{definition}
 $maj(w)= \sum_{lb_{i}(w)>0}{t_{i}}$, (where $lb_{i}(w)$ is the $i-th$ digit of the corresponding $lb$-vector).
 \begin{enumerate}
 	\item $t_{i}=1, t_{i+1}=0$ if $a_{i}=a_{i+1}=0$.
 	\item $t_{i}=2j$, if $a_{i}=j, a_{i+1}=\bar{j}$
 	\item $t_{i}=2j+1$, if $a_{i}=\bar{j}, a_{i+1}=j$
 	\item $t_{i}=0$, otherwise.
 \end{enumerate}
 \end{definition}
 
 \begin{example}: By the above definition, we see that the above $w(\pi)$ has maj as follows:
 $0+0+0+0+0+1+0+2(1)+0+0+0+0+0+0+0+0+2(3)+1=10=maj(\pi)$.
 And we have the theorem as an analogue of theorem 2.1(i) in \cite{MR1644459} 
 \end{example}
 \begin{theorem}
 	Let $f:S_{B}(\langle n\rangle, k)\mapsto SR_{n,k}$ be the above bijection in between $S_{B}(\langle n\rangle, k)$ and $SR_{n,k}$. Then for any $\pi \in S_{B}(\langle n\rangle, k), maj (f(\pi))= maj(\pi)$.
 \end{theorem}
 We define an analogue of ls vector for SRGF as follows:
 \begin{definition}
 If $w=w(\pi)= a_{0}a_{1}a^{*}_{1}a_{2}a^{*}_{2}\cdots a_{k}a^{*}_{k}$ is an SRGF, then 
 \begin{enumerate}
 	\item $ls_{B}(a_{l})=0$, if $a_{l}=0$, or $\bar{j}$ for some $j>0$. 
 \item Each pair $\bar{j}j$, $j$, contributes $2j-1$ and each pair $j\bar{j}, j$ contributes $2j-2$, if $j>0$.
 \item $ls_{B}$ statistics is the sum of the digits in the $ls_{B}$ vector.
 \item  Adding the digits of $ls_{B}(w(\pi))$, we get $los_{B}(\pi)$ always.
\end{enumerate}
\end{definition}
\begin{example}If we consider again $\pi = 0\bar{1}13\bar{3}/\bar{2}4/2\bar{4}/\bar{5}/5/\bar{6}\bar{8}/68/\bar{7}/7$, then the $ls_{B}$ vector is
 $00001000(=2(1)-2)003(=2(2)-1)05(2(3)-1)07(=2(4)-1)05(2(3)-1)$. If we add the digits we get $21$ which is same as the $los_{B}(\pi)$.
 \end{example}
 
 \begin{definition} We define a bijection between $S_{B}(\langle n \rangle, k)$ and $RR(\langle n \rangle, k)$, where $RR(\langle n \rangle, k)$, is the set of all $(2k+1)X(2n+1)$ row echelon form matrices where
 \begin{enumerate}
 	\item Every entry is either $0$, or $\bar{1}$, or $1$ and the first entry in the first row is always $1$ as, $0\in S_{0}$ and we choose keeping the $\bar{1}$ in the column prior to that of $1$.
 	\item After that, in the first row $1$ and $\bar{1}$ appear as a pair (since in $S_{0}$ positive and the corresponding negetive integer appears as a pair) always $\bar{1}1$ is placed as a pair in consequtive columns.
 	\item There is at least one $1$ or one $\bar{1}$ in each row and exactly one $\bar{1}$ or one $1$ in each column.
 	\item Excluding the first row and first column, if we have $\bar{1}$ (or $1$) in the row $2i$ and column $j$, then we have $1$ (or $\bar{1}$ in the row $2i+1$ or $2i-1 (i>1)$ and in column $j+1$ or in $j-1 (j>2)$ and conversely.
 	\item  Due to the standardization condition, always a pair $\bar{1}1$ in two consecutive rows and columns respectively appear in some prior columns before the pair $1\bar{1}$, if the second pair belong to two same consecutive rows as in the prior $\bar{1}1$.
 	\item  In any column $j>1$, the leading non zero element is $1$, if the row $i$ is odd and the leading non zero element is $\bar{1}$, if the row $i$ is even (as $m_{i} \in S_{2i}$ always).
 \end{enumerate}
 \end{definition}
 
 Next we define an analogue of six more vectors in case of SRGF corresponding the type B analogues of six more statistics defined by Steingrimsson as follows. Consider the SRGF $w=a_{0}a_{1}a^{*}_{1}a_{2}a^{*}_{2}\cdots a_{k}a^{*}_{k}$. 
 \begin{definition}
  An analogue for rcb vector in case of SRGF is as follows:
 \begin{enumerate}
 	\item Because of the condition $s>0$, in the definition of $rcb_{B}$, we set $rcb_{B}(a_{0})=0$, and $rcb_{B}(\bar{j})=0$ for any $j>0$.
 	\item If we have an occurrence of $00$ after the first $a_{0}=0$ in the SRGF, each $0$ contributes $m$, where $m$ is the number of  $j$ to it's right, so that $|j|>0$.
 	\item For any $\bar{j}j (j>0)$,  $rcb_{B}(j)=2m$, where $m$ is the number of $i$ to the right of $j$, so that $i>j$.
 	\item For any $j\bar{j}$ afterwards, we set $rcb_{B}(j)=2m+1$ where $i$ is as before.
 \end{enumerate}
 \end{definition}
 	 	
 	\begin{example} If $\pi =0\bar{1}13\bar{3}8\bar{8}/\bar{2}4/2\bar{4}/\bar{5}/5/\bar{6}\bar{8}/68/\bar{7}9/7\bar{9}$, then the $rcb_{B}$ vector is\\
 	$0080608700402000210$. And if we add the digits we get $38$ which is same as the $rcb_{B}(\pi)$.
 	\end{example}

 	\begin{definition} An analogue for lcb vector in case of SRGF as follows:
 	\begin{enumerate}
 		\item 	Define $lcb_{B}(a_{l})=0$, if $a_{l}=a_{0}$ or negetive. 
 	\item  If $a_{l}=j>0$, then that $a_{l}$ does contribute $0$, for each pair of zeros to it's right and if we have a pair $\bar{j}j$ and no smaller positive element to it's right, then that $j>0$ contributes $1$.
 	\item  Any pair $\bar{j}j$, the $j>0$, gives $2$ for each pair $\bar{i}i$ and/or $i\bar{i}$ to it's right for each $i<j$. Additionally, $j>0$, gives $1$ for itself in that case. 
 	\item Any pair, $j\bar{j}$, the $j>0$ gives $2l$, where $l$ is the number of $i's$ to it's right $i>0, i<j$.
 \end{enumerate}
 \end{definition}
 
 \begin{example} If $\pi =0\bar{1}13\bar{3}/\bar{2}4/2\bar{4}/\bar{5}/5/\bar{6}\bar{8}/68/\bar{7}/7$, then the $rs$ vector is\\
 	$00001000001010301$. If we add the digits we get $7$ which is same as the $lcb_{B}(\pi)$.
 	\end{example}
 	
 	\begin{definition}An analogue for $rob_{B}$ vector in case of SRGF as follows:
 	\begin{enumerate}
 	\item Define
 	$rob_{B}(a_{l})=0$, if $l=0$ or if $a_{l}$ is negetive. 
 	\item  If we have a pair $00$, then we consider the first $0$ gives $0$ and the second one gives $2l$, where $l$ is the number of distinct $i>0$ to the right of that $0$, that does not appear to the left of that $0$. 
 	\item If there is a pair $\bar{j}j$ or $j\bar{j}, j>0$, then the $j>0$ gives $2l$, where $l$ is the number of distinct $i>j$ to the right of $j$, that are not to the left of that $j$.
 	 	\end{enumerate}
 	 	\end{definition}
 	 	\begin{example}
 If $\pi =0\bar{1}13\bar{3}/\bar{2}4/2\bar{4}/\bar{5}/5/\bar{6}\bar{8}/68/\bar{7}/7$, then the $rob_{B}$ vector is\\
 $0080606600402000000$. If we add the digits we get $32$ which is same as the $rob_{B}$ statistics of that type $B$ partition.
 \end{example}
  \begin{definition}
 An analogue for lob vector in case of SRGF as follows:
 \begin{enumerate}
 	\item  If $a_{l}=0$ or negetive, then $lob_{B}(a_{l})=0$. Otherwise, for the first occurrence of $\bar{j}j(j>0)$, the $j>0$ gives $1$.
 	\item  If there is any $\bar{j}j$ or $j\bar{j}$ repeated for the same $j$, that does not contribute anything.
 \end{enumerate}
 \end{definition}
 \begin{example}
If $\pi =0\bar{1}13\bar{3}/\bar{2}4/2\bar{4}/\bar{5}/5/\bar{6}\bar{8}/68/\bar{7}/7$, then the $lob_{B}$ vector is\\ $00001000001010100$ and if we add the digits, then we get $4$ which is the same as the $lob_{B}$ statistics for that $\pi$.
\end{example}

 \begin{definition}  An analogue for rcs vector in case of SRGF is as follows:
 \begin{enumerate}
  	\item We define $rcs_{B} (a_{l})=0$, if $a_{l}=0$ or negetive.
 	\item The first occurrence of $\bar{j}j$ does not contribute anything. 
 	\item If $\bar{j}j$ repeats after $\bar{j}j$, then the $j>0$, gives $2(l-m)$, where $l$ is the number of pairs $\bar{i}i$ and/or $i\bar{i}$ to the left of $\bar{j}j$ with $i<j$ and $m$ is the number of pairs $\bar{i}i$ and/or $i\bar{i}$ to the left of $\bar{j}j$ with $i>j$.
 	\item If $j\bar{j}$, after the $\bar{j}j$ appears, then the $j>0$ gives $1+2(l-m)$, where $l,m$ are as before. 
\end{enumerate}
\end{definition}

\begin{example} If $\pi =0\bar{2}2/\bar{1}7/1\bar{7}/\bar{3}\bar{6}/36/\bar{4}5/4\bar{5}$, then the $rcs_{B}$ vector is $000000000100250$.
\end{example}
  \begin{definition}An analogue for lcs vector in case of SRGF as follows:
 \begin{enumerate}
  	\item $lcs_{B}(a_{l})=0$, if $a_{l}=0$ or negetive. The pair $1\bar{1}$ does not contribute anything. If we have repeated occurrence of $\bar{1}1$, the right most $1$ contributes $1$ provided there is no $\bar{1}$ to it's right, otherwise that $1$ contributes $0$.
 	 	\item Consider any $\bar{j}j, j>1$. Then $j$ contributes $2l+1$, where $l$ is the number of distinct $i$ to the left of $j$, that are not to the right of $j$, such that $0<i<j$, provided there is no $j$ to the right of that $j$ again.
 	 	\item If $j>1$ is like above, then it contributes $2l$, where $l$ is as above, provided there is some (possibly more than $1$) $j$ to the right of that $j$.
 	 	\item Consider any $j\bar{j}, j>1$. Then, that $j$ contributes $2l$, where $l$ is as above.
 	  	 	 	
 \end{enumerate}
 	\end{definition}
\begin{example} Thus for example if $\pi =0\bar{1}13\bar{3}/\bar{2}4/2\bar{4}/\bar{5}/5/\bar{6}\bar{8}/68/\bar{7}9/7\bar{9}$, then the $lcs_{B}$ vector is $0000000000304040560$, adding the digits we get $22$ which is the $lcs_{B}$ statistics for the partition $\pi$.
 \end{example}
  
 As in \cite{MR1644459} we can create an analogous bijection $h:S_{B}(\langle n\rangle,k)\mapsto RR(\langle n\rangle,k)$ as follows $h(\pi)=M$, where $M= (m_{i,j})_{(2k+1)X(2n+1)}$, with 
 \begin{enumerate}
 	\item  $m_{1,1}= 1$ (as $0\in S_{0}$)
 	\item $m_{1,j}=\bar{1}, m_{1,j+1}=1$, if $\bar{j}, j\in S_{0}\forall j\in \{1,2,\cdots 2n+1\}$
 	\item For $i\in \{1,2,\cdots 2k+1\} 1,m_{ij} =\bar{1}$ if $\bar{j}\in S_{i-1}$ and $m_{ij} =1$, if $j\in S_{i-1}$
 	 \end{enumerate}
 
 Now if we define the major index of such a matrix $M$ as $maj(M) =\sum_{m_{i,j}=1} i$, where the sum is restricted to those $1$'s which have another $1$, strictly to their south-west (as in \cite{MR1644459}) then we have the following theorem as an analogue of theorem 3.3 (i) as follows:
 
 \begin{theorem}
 	For the above bijection $h$, and for any $\pi\in S_{B}(\langle n \rangle, k)$ $maj(h(\pi))= maj (\pi)$
 \end{theorem}
 For example for our $\pi = 0\bar{1}13\bar{3}/\bar{2}4/2\bar{4}/\bar{5}/5/\bar{6}\bar{8}/68/\bar{7}/7, h(\pi)$ is a $9X17$ matrix as follows: 
 	\[ \begin{pmatrix}
 	1 & \bar{1} & 1 & 0 & 0 & \bar{1} & 1 & 0 & 0 & 0& 0 & 0& 0 & 0& 0 & 0 & 0\\
 	0 & 0 & 0 & \bar{1} & 0 & 0 & 0 & 0 & 1 & 0& 0 & 0& 0 & 0& 0 & 0 & 0\\
 	0 & 0 & 0 & 0 & 1 & 0 & 0 & 0 & 0 & \bar{1}& 0 & 0& 0 & 0& 0 & 0 & 0\\
 	0 & 0 & 0 & 0 & 0 & 0 & 0 & 0 & 0 & 0& \bar{1} & 0& 0 & 0& 0 & 0 & 0\\
 	0 & 0 & 0 & 0 & 0 & 0 & 0 & 0 & 0 & 0& 0 & 1& 0 & 0& 0 & 0 & 0\\
 	0 & 0 & 0 & 0 & 0 & 0 & 0 & 0 & 0 & 0& 0 & 0& \bar{1} & 0& 0 & \bar{1} & 0\\
 	0 & 0 & 0 & 0 & 0 & 0 & 0 & 0 & 0 & 0& 0 & 0& 0 & 1& 0 & 0 & 1\\
 	0 & 0 & 0 & 0 & 0 & 0 & 0 & 0 & 0 & 0& 0 & 0& 0 & 0& \bar{1} & 0 & 0\\
 	0 & 0 & 0 & 0 & 0 & 0 & 0 & 0 & 0 & 0& 0 & 0& 0 & 0& 0 & 1 & 0\\
 	
 	\end{pmatrix} \]

Let us define the dual descent multiset analogously as in \cite{MR1644459} in case of type B partitions.

\begin{definition}
	For any $\pi\in S_{B}(\langle n\rangle, k)$ the dual descent set of $\pi$ is denoted as $\widehat {Des\pi}$ and is defined as the multiset$\{2^{a_{2}} 3^{a_{3}},\cdots (2k+1)^{a_{2k+1}}\}$
	where $a_{i}$ is the number of $s\in S_{i}$ such that $s>m_{i-1}, \forall i \in \{1,2,\cdots 2k+1\}$
	\end{definition}
	
	Accordingly, we define the natural analogue of dual major index for any type B partition as $\widehat{maj_{B}{\pi}}=\sum_{i\in \widehat {Des\pi}}(i-1)$.
	
	Afterwards, we find a reccurrence relation for the generating function of the dual major index for type B partitions as follows:
	\begin{theorem}
		$\widehat{S_{B}}[n,k]=q^{2k}\widehat{S_{B}}[n-1,k-1]+[2k+1]\widehat{S_{B}}[n-1,k]$, where $\widehat{S_{B}}[n,k]$ is the generating function of $\widehat{maj_{B}{\pi}}$
		\end{theorem}
		Questions:\begin{enumerate}
			\item Does the analogous result follow for dual major index for the set of standard type B partitions as in \cite{MR1644459}?
			\item Is there any way to define analogue of $r-maj$ index for standardized type B partitions, and finding the corresponding generating functions as in \cite{MR1644459} ?
			\item Is there any analogue of $p,q$-Stirling number (first introduced by Wachs and White) of second kind for type B partitions and a way to find the generating function of the corresponding joint distributions as in \cite{MR1644459}?
		\end{enumerate}

Acknowledgment: Author thanks to Bruce Sagan for introducing this area to work on and for providing several suggestions, multiple advice and  helpful discussions on this topic via many communications.
  Author thanks to Jordan O Tirrell towards providing valuable suggestions in order to write this paper.
  Author thanks to Tucker Ervin for providing many great discussions, further ideas, checking the work for corrections, extending helpful advice through several communications.
Author thanks to David Gajewski for helping in LaTeX. This work is partially supported via travel grants from NSF and NSA to travel in conferences in Pattern in Permutations.

\end{document}